\documentclass[11pt]{amsart}
\usepackage{graphicx,color}
\usepackage[left=1in,right=1in,top=1in,bottom=1in]{geometry}
\usepackage{latexsym}
\usepackage{cite}
\usepackage{amsmath,amssymb,amsthm,mathrsfs}

\allowdisplaybreaks[1]  

\newtheorem{theorem}{Theorem}

\newtheorem{corollary}[theorem]{Corollary}

\newtheorem{remark}[theorem]{Remark}

\renewcommand{\phi}{\varphi}

\let\epsilon=\varepsilon

\def\crn#1#2{{\vcenter{\vbox{
\hbox{\kern#2pt \vrule width.#2pt height#1pt
 }
\hrule height.#2pt}}}}

\newcounter{mnotecount}[section]
\let\oldmarginpar\marginpar
\renewcommand\marginpar[1]{\-\oldmarginpar[\raggedleft\footnotesize #1]%
{\raggedright\footnotesize #1}}

\begin{document}

\title[Introduction to RG-2 Flow]{A geometric introduction to the 2-loop renormalization group flow}
\author{Karsten Gimre}
\address[Gimre]{Department of Mathematics, Columbia University, New York City, New York}
\email{gimre@math.columbia.edu}
\author{Christine Guenther}
\address[Guenther]{Department of Mathematics and Computer Science, Pacific University,
Forest Grove, Oregon, 97116}
\email{guenther@pacificu.edu}
\author{James Isenberg}
\address[Isenberg]{Department of Mathematics, University of Oregon, Eugene, Oregon}
\email{isenberg@uoregon.edu}

\thanks{KG is partially supported by the NSF under grant DGE-1144155.}
\thanks{CG is partially supported by the Simons Foundation Collaboration Grant for Mathematicians - 283083}
\thanks{JI is partially supported by the NSF under grant PHY-1306441.}

\date{\today}

\begin{abstract}
The Ricci flow has been of fundamental importance in mathematics, most famously though its use as a tool for proving the Poincar\'e Conjecture and Thurston's Geometrization Conjecture. It has a parallel life in physics, arising as the first order approximation of  the Renormalization Group flow for the nonlinear sigma model of quantum field theory. There recently has been interest in the second order approximation of this flow, called the RG-2 flow, which mathematically appears as a natural nonlinear deformation of the Ricci flow. A  curvature flow arising from quantum field theory seems to us to capture the spirit of Yvonne Choquet Bruhat's extensive work in mathematical physics, and so in this commemorative article we give a geometric introduction to the RG-2 flow. 

A number of new results are presented as part of this narrative:  short-time existence and uniqueness results in all dimensions if the sectional curvatures $K_{ij}$ satisfy certain inequalities; the calculation of fixed points for $n = 3$  dimensions; a reformulation of constant curvature solutions in terms of the Lambert W function; a classification of the solutions that evolve only by homothety; an analog for RG-flow of the 2-dimensional Ricci flow solution known to mathematicians as the cigar soliton, and discussed in the physics literature as Witten's black hole. We conclude with a list of Open Problems whose resolutions  would substantially increase our understanding of the RG-2 flow  both physically and mathematically.
\end{abstract}

\maketitle

\section{Introduction}

Yvonne Choquet-Bruhat has made fundamental contributions to the study of both hyperbolic  PDEs  (notably her work in proving the well-posedness of the Cauchy problem for Einstein's gravitational field equations) and elliptic PDEs (we cite here her foundational work on the conformal method for studying the Einstein constraint equations). We (JI) have had the delightful privilege of knowing Yvonne for over 40 years, of working with her on a wide variety of projects over the last 35 years, and of counting her as a close friend. In the spirit of this collaboration and friendship, and with an eye on Yvonne's wonderful adventurousness, we offer here a brief survey  of a geometric heat flow of relatively recent interest which is related to Ricci flow, and has its roots in physics. 

During the 1960's, Gell-mann and Levy introduced a class of quantum field theories which they labeled ``nonlinear sigma models" \cite{GL}. These quantum field theories are based on maps  $\phi: \Sigma \rightarrow M$ (for $(\Sigma,\gamma)$ a Riemannian surface, and for $(M,g)$ a Riemannian manifold) and the corresponding harmonic map functional
\begin{equation} 
 S(\phi) = \int_\Sigma \frac{1}{\alpha} g_{ij} (\phi(x) ) \partial^\mu\phi^i(x)\partial^\nu \phi^j(x)\gamma_{\mu \nu} dx.
 \end{equation}
 ($\alpha$ is a positive constant.)
The analysis of the nonlinear sigma model quantum field theories involves the choice of certain ``cut-offs", and it is important for the plausibility of the theories as tools for physical prediction that their predictions not change under rescaling of the cutoffs. To achieve this, physicists generally deform the metric $g$ along with the deformation of the cutoff scale. Using $t$  to parametrize these deformations, physicists call the evolution of $g(t)$ the ``renormalization group flow." Renormalization is an essential concept in theoretical physics; see \cite{D} for a short  intuitive introduction, and see \cite{F2} for an overview of the physics of the renormalization group flow for the nonlinear sigma model. 
A comprehensive introduction can be found in the lecture notes  \cite{IAS},\cite{G}. Despite its importance in physics, the renormalization group flow does not appear in the literature as an explicit PDE system for the evolution $g(t)$ of the metric. Rather, what one finds are perturbative (in $\alpha$) expansions of the flow system (see \cite{F1}). The first order truncation produces the celebrated Ricci flow
\begin{equation}
\frac{\partial }{\partial t}g=-2Rc,
\end{equation}
where $Rc$ is the Ricci curvature tensor,  the trace of the full curvature tensor. (This in turn can be thought of as the collection of the Gaussian curvatures of 2-dimensional subspaces of the manifold.

The Ricci flow was introduced to the mathematics community in 1982 by Richard Hamilton, who laid out a program to use it to verify the Thurston's Geometrization Conjecture. Although his program was essentially completed by Perelman in 2003, the study  of Ricci flow  continues to be an extremely active area of mathematics \cite{C, Cetal, Cetala}, and its effectiveness as a tool for attacking important problems in geometry has been remarkable (see for example \cite{BW, BrSc}). Interestingly, one of Perelman's key results is the demonstration that the Ricci flow is a gradient flow, and he credits his energy functional formulation of this flow to ideas arising in the quantum field theory literature.

Physicists have used  the Ricci flow as an approximation of the full renormalization group flow. The question of whether or not this is a good approximation naturally arises, and one way to address this issue is to study the expansion truncated at the second term; one obtains the 2-loop renormalization group flow equation (which we label ``RG-2 flow"), which takes the following form:
 \begin{equation}
 \label{RG2flow}
\frac{\partial }{\partial t}g=-2Rc-\frac{\alpha }{2}Rm^{2}.
\end{equation}
Here, the quadratic Riemann term on the right hand side is given explicitly (in index form) by
\begin{equation}
\label{Rm2}
Rm^2_{ij} = Rm_{iklm} Rm_{jpqr} g^{kp} g^{lq} g^{mr}, 
\end{equation} 
so this flow involves the full curvature tensor.  We note that this $Rm^2$ term is  familiar to mathematicians, arising in the study of Einstein manifolds (see \cite{B}, Definition 1.131, and also \cite{Ber}). 

Only recently has there been much attention given to the mathematical nature of the RG-2 flow system and its solutions. This is partly because, unlike other well-studied geometric heat flows such as Ricci flow and mean curvature flow, the RG-2 flow is not a weakly parabolic system. It \emph{is} the case that RG-2 flow, like the others, is diffeomorphism-covariant in the sense that if $g(t)$ is a solution with initial condition $g(0) = g_0$, then $\phi^*g(t)$ is also a solution (since $Rc(\phi^*g) = \phi^*Rc(g)$ and $Rm^2(\phi^*g) = \phi^*Rm^2(g)$).
As a consequence of this feature, none of these flows are \emph{strictly} parabolic. For Ricci flow, either by restricting diffeomorphism freedom or by working with the diffeomorphically-related DeTurck flow, one can effectively turn the flow equations into a parabolic PDE system. This does not work for RG-2 flow. One sees this by noting that the imposition of a coordinate condition like ``harmonic coordinates", which results in Ricci flow appearing to take parabolic form, does not do the same for RG-2 flow. 

There are special cases in which RG-2 flow is weakly parabolic: Oliynyk shows in \cite{O} that in two dimensions the RG-2 flow is (weakly) parabolic assuming $1 +  \alpha K \ge 0$, where $K$ is the Gaussian curvature. For surfaces of  Euler characteristic $\chi < 0,$ he finds subspaces of the space of smooth Riemannian metrics that are invariant under the flow, and for which the flow stays parabolic (or, alternatively, backwards parabolic).  Cremaschi and Mantegazza prove the analogous parabolicity result in 3 dimensions (see \cite{CM}), where the curvature condition is $1+\alpha K_{ij} > 0$ for all sectional curvatures $K_{ij}.$ We have recently proved that an analogous result holds in all dimensions (see \cite{GGIa}): 

\begin{theorem}\label{ndim} Let $(M, g_0)$ be a closed $n$-dimensional Riemannian manifold. If $1+\alpha K_{ij} > 0$ for all sectional curvatures $K_{ij}$, then there exists a unique solution $g(t)$ of the initial value problem $\frac{\partial }{\partial t}g=-2Rc-\frac{\alpha }{2}Rm^{2}$,  $g(0) = g_0$,  on some time interval $[0,T).$ 
\end{theorem}

A superficial inspection of the RG-2 flow equation \eqref{RG2flow}-\eqref{Rm2} suggests that the second term is important for a given solution if and only if the norm of the curvature approaches the value $\frac{2}{\alpha}$. The results noted above concerning weak parabolicity support this contention. In our discussion of various special solutions below, we note further concrete evidence for it, in the form of families of solutions (defined by isometries) for which a portion of the RG-2 solutions (those with positive curvature, or curvature not too negative with respect to $\frac{2}{\alpha}$) behave very similarly to solutions of Ricci flow, while a complementary portion of the RG-2 solutions (those with large negative curvature) behave qualitatively very differently.

The relative size of the curvature  and  $\frac{2}{\alpha}$ also determines whether the evolution of the volume element $d \mu$ is similar to or different from that of the Ricci flow.  The volume element evolves according to 
\begin{equation}
\label{volevoln}
\frac{\partial }{\partial t}d\mu =(-R-\frac{\alpha }{4}\left\vert
Rm\right\vert ^{2})d\mu ,
\end{equation}%
where $R$ is the scalar curvature (defined to be the trace of the Ricci curvature tensor).  If $R > 0,$ the volume decreases. In the $R < 0$ case, one can have $\frac{\partial}{\partial t}{d\mu}$ positive or negative (in contrast to the Ricci flow),  depending on the relative size of $\frac{2}{\alpha}$ and the curvature.  
Note that one can use \eqref{volevoln} to construct a volume-normalized version of RG-2 flow: 
\begin{equation}
\frac{\partial }{\partial t}g=-2Rc-\frac{\alpha }{2}Rm^{2} + \frac{1}{n} \frac{\int R + \frac{\alpha}{4}|Rm|^2 d\mu }{\int d\mu}.
\end{equation}
There do not appear to be any studies of this volume-normalized flow in the literature.

For a relatively new and relatively unstudied geometric heat flow system like RG-2 flow, one of the more useful ways to begin to understand it is to examine special classes of solutions. We do that here, starting in Section \ref{FixedPoints} with fixed point solutions. We proceed in Section \ref{ConstCurv} to review results concerning solutions of the RG-2 flow which have constant curvature. In Section \ref{Solitons} we discuss RG-2 solitons, and in Section \ref{LocHomog} we review results concerning RG-2 flow restricted to the class of locally homogeneous geometries. We conclude this work (see Section \ref{Open}) by stating a number of outstanding open problems concerning RG-2 flow and its solutions.

\section{Fixed Points of RG-2 Flow}  
\label{FixedPoints}

The fixed points of any dynamical system are those solutions which do not change in time; hence for RG-2 flow they are those geometries for which the right hand side of \eqref{RG2flow} vanishes: 
\begin{equation}
\label{fixpts}
Rc[g] =-\frac{\alpha}{4}Rm^2[g].
\end{equation}
Not surprisingly, in view of this equation, the fixed points which are known are metrics with negative Ricci curvature.

In dimension two, equation \eqref{fixpts} reduces to the condition that the curvature be constant and negative. Such geometries---the hyperbolic surfaces---are well understood. In dimension three, the analysis is a bit more involved, but we are still able to determine the complete set of RG-2 fixed points. The starting point for this analysis is to recall that in three dimensions, the Riemann curvature tensor can be written strictly in terms of the Ricci tensor. It follows that the RG-2 fixed point equation \eqref{fixpts} can be rewritten (in index form) as follows:
\begin{equation}
\label{3dfixpt}
R_{ij}=-\frac{\alpha}{2}(-R_i^s R_{sj} + R R_{ij} + |Rc|^2 g_{ij} - \frac{R^2}{2} g_{ij}).
\end{equation}

Since the Ricci curvature tensor $R_{ij}$ is a symmetric tensor, one can diagonalize it at a point, in terms of an orthonormal basis. Using  $(\lambda, \nu, \mu)$ to label the eigenvalues of the Ricci tensor under this diagonalisation, one can rewrite the fixed point condition \eqref{3dfixpt} as a system of three (quadratic) algebraic equations for these eigenvalues.:
\begin{align}
-2\lambda+ \frac{\alpha}{2}(\lambda + \mu + \nu)^2-\alpha \lambda (\lambda + \mu + \nu)-\alpha(\lambda^2+\mu^2+\nu^2)+\alpha\lambda^2= 0\\
-2\mu+ \frac{\alpha}{2}(\lambda + \mu + \nu)^2-\alpha \mu(\lambda + \mu + \nu) -\alpha(\lambda^2+\mu^2+\nu^2)+\alpha \mu^2 = 0 \\
 -2\nu + \frac{\alpha}{2}(\lambda + \mu + \nu)^2-\alpha \nu(\lambda + \mu + \nu) -\alpha(\lambda^2+\mu^2+\nu^2)+\alpha \nu^2 = 0. 
\end{align}
After some calculation, one finds that these equations admit the following four distinct solutions, and no others:
\begin{align*}
&1) \ \ \lambda = \mu = \nu = 0,\\
&2)\ \  \lambda = \mu = \nu = -4/\alpha  \\
&3)\ \  \{\lambda, \mu, \nu\} =\{-2/\alpha, -2/\alpha, \ 0\} \\
&4)\ \ \{\lambda, \mu, \nu\} = \{-4/\alpha, -2/\alpha, -2/\alpha\}.
\end{align*}

Each of these four solutions for $(\lambda, \nu, \mu)$ specifies a set of constant values for the three Ricci tensor eigenvalues. Since the Ricci curvature tensor diagonalisation can be done at any point, each of these consequently specifies a geometry (or set of geometries)  of  specified constant Ricci curvature. Hence we obtain four distinct (sets of) fixed point geometries for RG-2 flow. We contrast this with the set of fixed points for Ricci flow, which (in all dimensions) comprise the Ricci flat geometries only.

Constant Ricci curvature geometries are known as ``Einstein spaces". In three dimensions, since the Ricci curvature controls the full geometry, they are ``constant curvature homogeneous spaces". Noting that in three dimensions,  the eigenvalues of the sectional curvatures are given by $\frac{1}{2} ( \mu + \nu - \lambda,  \nu + \lambda -\mu,  \nu + \lambda -\mu), $ we readily determine that the fixed points 1), 2), and 3) correspond to $\mathbb R^3,$ $\mathbb H^3$, and $\mathbb H^2 \times \mathbb R.$ The fixed point described by $\{\lambda, \mu, \nu\} = \{-4/\alpha, -2/\alpha, -2/\alpha\}$ is less understood. We note in particular that it follows from Proposition 5.1 in \cite{KN}\footnote{Proposition 5.1: A locally homogeneous Riemannian 3-manifold with principal Ricci curvatures $\lambda, \mu, \nu$ exists if the curvatures satisfy at least one of the following (partially overlapping) three sets of conditions: i) All are equal, or two are equal and the last is zero; ii) $\lambda\mu\nu>0$, or at least two of them are zero; iii)all are non-positive, at most one is zero, and up to re-numeration satisfy the inequalities $2\lambda<\mu+\nu,$ $\lambda(\mu+\nu)\le \mu^2+\nu^2$.} that geometries satisfying this condition cannot be locally homogeneous.

Are these fixed points of RG-2 flow stable? Results in \cite{GO} show that at least some of them are. In that work, maximal regularity techniques are used to show that, as is the case for Ricci flow, the flat torus (fixed point 1) is stable, and in particular there exists an exponentially attractive center manifold composed entirely of flat metrics. Also in that paper, it is shown that given {\it any} hyperbolic space of any constant curvature $K < 0$, one can modify the RG-2 flow by scaling so that the space is a fixed point of the modified flow, and one can show stability if $\alpha K$ is not too negative.

\begin{remark} As noted previously, to be physically meaningful one requires $\alpha>0$. However, mathematically one can assume $\alpha <0$. From \cite{KN}, one sees that the fixed point $\{-4/\alpha, -2/\alpha, -2/\alpha\}$ in the $\alpha<0$ case is indeed locally homogeneous, and one also recovers $\mathbb S^3$ and $\mathbb S^2 \times \mathbb R$ as fixed points under this assumption. \end{remark}

\section{Constant Curvature Solutions} 
\label{ConstCurv}

Among the simplest collection of solutions to the RG-2 flow, besides the fixed points, are those with constant curvature. We recall that it follows from diffeomorphism invariance and the uniqueness of solutions for the Ricci flow initial value problem that isometries of the initial metric $g_0$ are preserved along Ricci flow solutions $g(t)$. For RG-2 flow, diffeomorphism invariance holds, but uniqueness has not been generally proven (see Theorem \ref{ndim} for a special case where it is known). Hence, there is no proof that isometries are preserved along RG-2 flow. We can, however, focus on those constant curvature solutions such that indeed the constant curvature condition is preserved. This leads us to consider solutions of the form 
\begin{equation}
\label{constcurvat}
g(t)=\phi(t)g_K,
\end{equation}
where $g_K$ denotes a fixed constant curvature metric (of curvature $K$). 

As is shown in \cite{GGI}, the RG-2 flow for evolving metrics of the form \eqref{constcurvat} implies that
$\phi(t)$ must satisfy the equation 
\begin{equation}
\label{evolvephi}
\phi'(t) = -2K(n-1) - \frac{\alpha}{\phi(t)} K^2 (n-1).
\end{equation} 
One readily determines that this ODE is solved implicitly by functions $\phi(t)$ which satisfy 
\begin{equation}
\label{imphi}
\phi(t) = -2K(n-1)t + 1 + \frac{\alpha K }{2} \ln \Big |\frac{2 \phi(t) + \alpha K}{2 + \alpha K}\Big |,
\end{equation}
where $n$ denotes the dimension of the space. 

Using this implicit expression, one may seek times $T$ such that $\phi(T)$ vanishes: one verifies that for  $K$ not equal to either zero or $-\frac{2}{\alpha}$ (these values correspond to RG-2 fixed points, as discussed above), 
one has $\phi(T) = 0$ for 
\begin{equation}
\label{Tvalues}
T = \frac{1}{2K(n-1)}+\frac{\alpha}{4(n-1)}\ln \Big | \frac{\alpha K}{2+\alpha K} \Big|.
\end{equation}
For those values of $K$ for which this expression for $T$ is a positive real (finite) number, we see that the corresponding constant curvature RG-2 solution becomes singular in finite time. To survey what in fact holds for various choices of $K$, it is useful to note that the implicit solutions $\phi(t)$ of \eqref{imphi} are given by  the ``Lambert W function",  denoted by $W$, which is defined to be the (multivalued) inverse of the map $f(W(z))= W(z)e^W(z)$, with $W_0$ and $W_{-1}$ denoting specific branch choices (see for example \cite{CG}). We now list the behavior of the RG-2 flow for the various choices of the constant curvature $K$:
\begin{enumerate}
 
\item  $K = 0$: This is a fixed point in all dimensions.

\item  $K > 0$:  The constant curvature sphere collapses to a point in finite time, with the time to extinction (given by the value of $T$ in \eqref{Tvalues}, which for $K>0$ is positive real) less than the corresponding collapse time for spheres evolving under Ricci flow. As an explicit example,  if $n = 3$, $\alpha = 1$ and $K = 1,$ one has the solution $\phi(t) = -1/2 - 1/2 W_{-1}( -3e^{8t-3})
.$

\item $K<0$, and $2+\alpha K < 0$: Constant negative curvature spaces, with curvature very negative, collapse to a point in finite time $T$.  For example letting $n = 3$, $\alpha = 1$ and $ K = -6,$ one has the function $\phi(t) = 3+3 W_0(2/3 e^{i \pi + 8t - 2/3})$, which is  real until the extinction time $T$ (given by \eqref{Tvalues} above). This behavior of  hyperbolic geometries under RG-2 flow differs from that under Ricci flow, in that the latter continues to expand for all time.

\item $K<0$, and $2 + \alpha K > 0$: For constant negative  curvature spaces with small curvature, the RG-2 flow geometry initially expands, and in fact continues to expand for all time, just as is the case under Ricci flow.  For example, if $n = 3$$\alpha = 1$ and $K = -1/2$ one has the solution $\phi(t) = 1/4 + 1/4 W_0(3e^{8t+3}),$ which behaves monotonically.  

\item  $K < 0$ and $2 + \alpha K = 0.$ This is a fixed point in all dimensions. In dimension 3, this corresponds to the case $Rc = (-4/\alpha,-4/\alpha,-4/\alpha)$, since $Rc = K(n-1) g$. 

\end{enumerate}

Note that for any fixed value of the coupling constant $\alpha$, if one chooses the value of $K$ to be sufficiently negative,  the behavior under the RG-2 flow differs from that under Ricci flow.

\section{Solitons}
\label{Solitons}

Fixed point solutions of RG-2 flow do not evolve, and the constant curvature solutions discussed above evolve only by scaling.  More generally, solutions that evolve only by scaling and diffeomorphism (i.e., those which evolve self-similarly),  are called {\it solitons}. For $\phi'(0) < 0$ they are called {\it shrinking solitons}, for $\phi'(0) > 0$ they are {\it expanding solitons}, and for $\phi'(0) = 0$, they are called {\it steady solitons}. In the study of Ricci flow and other geometric heat flows, 
solitons have proven to be very useful both for modeling the behavior of singularities and for modeling long-time (non-singular) behavior of solutions of the flows. 
Indeed, a commonly used  technique for analyzing and classifying  the geometry of singularities in Ricci flow solutions  is to generate a sequence of flow solutions  by scaling in space and time about the singularity, and then studying the limits of these sequences and relating them to solitons. 
Solitons have also been used to find so-called Harnack inequalities, which can be very helpful in comparing 
solution values at different times. Harnack inequalities are generally sharp for solitons.  

While RG-2 solitons have only recently been considered, and much is unknown concerning them, we have learned some facts about them, as surveyed here.

\subsection{Solitons with No Diffeomorphism Action}

Since the RG-2 equation is not scale-invariant, the possible soliton solutions with $\sigma(t)$ not constant and with no evolution by diffeomorphism---i.e., those for which  $g(t) = \sigma(t) g_0$ for some specified initial metric\footnote{We emphasize that here $g_0$ is a general initial metric, not necessarily one of constant curvature as in Section \ref{ConstCurv}.} ---are very restricted. We express this in the following classification theorem:

\begin{theorem}
\label{SolNoDiff}
Let $(M, g_0)$ be a closed Riemannian manifold. If $g(t)$ is a solution of the RG-2 flow such that $g(t) = \sigma(t) g_0$, with $\sigma(t) > 0$ not constant, then  $(M, g_0)$ is an Einstein manifold and $Rm^2 = \frac{1}{n} |Rm|^2 g$.
\end{theorem}

\begin{proof}
Calculating the right hand side of the RG-2 flow equation for metrics of the form $g(t) = \sigma(t) g_0$, we obtain
\begin{equation}
-2Rc(g(t)) - \alpha Rm^2(g(t)) = -2 Rc(g_0) - \frac{\alpha}{\sigma(t)} Rm^2(g_0).
\end{equation}
It then follows  that $g(t)$ is an RG-2 flow solution if and only if $\sigma(t)$  and $g_0$ satisfy the following tensor equation:
\begin{equation}
\label{sigmaeqn}
-2 Rc(g_0) - \frac{\alpha}{\sigma(t)} Rm^2(g_0) = \sigma'(t) g_0.
\end{equation}

Since the tensors $Rc$ and $Rm^2$ are both symmetric, we can diagonalize the left hand side of \eqref{sigmaeqn} with respect to an orthonormal basis. Labeling the eigenvalues (with respect to this diagonalization) of $Rc$ as $(a_1, a_2, a_3)$ and those of $Rm^2$ as $(b_1, b_2, b_3)$, we obtain the system
\begin{align}
& \ \  a_1+  \frac{\alpha}{\sigma(t)} b_1= \sigma'(t).\\
&\ \   a_2+  \frac{\alpha}{\sigma(t)} b_2= \sigma'(t)  \\
&\ \  a_3 +  \frac{\alpha}{\sigma(t)} b_3= \sigma'(t).
\end{align}
If we now subtract, for example, the first two equations, we have
$$ (a_1 - a_2) = (b_2 - b_1) \frac{\alpha}{\sigma(t)}.$$ 
If $\sigma(t) > 0$ is not constant, we must have $a_1 = a_2$ and $b_1 = b_2.$ Arguing similarly, we determine  that $a_1 = a_2 = a_3$, and $b_1 = b_2 = b_3.$ 

Thus $Rc = f_1(p) g$, for some function $f_1:M \rightarrow \mathbb R$, and it is then a consequence of Schur's Lemma\footnote{See \cite{P}. Schur's Lemma: On a Riemannian manifold of dimesions $n \ge 3$, if $Rc(v) = (n-1) f(p) v$ for all $v \in T_pM, p\in M$, then the metric is Einstein.} that the metric is Einstein.
We also have $Rm^2_{ij} = f_2(p) g_{ij}$ for some function $f_2:M \rightarrow \mathbb R$. Taking the trace of both sides yields $|Rm|^2 = n f_2(p)$; hence $Rm^2_{ij} = \frac{1}{n}|Rm|^2g_{ij}.$

\end{proof}

\begin{remark}
For Einstein manifolds in dimension 4, it is always true that $Rm^2 =\frac{1}{4} |Rm|^2 g$  (see \cite{B}, Remark 1.133). For  higher dimensions,   the condition holds if and only if the metric is critical for the functional $$S[Rm(g)] = \int_M |Rm|^2 d\mu$$
(see Corollary 4.72 in \cite{B}).
\end{remark}
\begin{remark}
We note that if one takes the divergence of both sides of the equality $Rm^2_{ij} =\frac{1}{n} |Rm|^2 g_{ij}$, one obtains the Bianchi-type identity 
 $$n \ div(Rm^2)_j = \nabla_j (tr Rm^2).$$
\end{remark}

Since a three-dimensional Riemannian manifold $(M,g)$ is Einstein if and only if it has constant curvature, we obtain the following as an immediate consequence of Theorem \ref{SolNoDiff}.

\begin{corollary}
Let $(M, g_0)$ be a closed 3-dimensional Riemannian manifold. Then there exists a solution  $g(t)$ of the RG-2 flow such that $g(t) = \sigma(t) g_0$, with $\sigma(t) > 0$ not constant, if and only if $M$ has constant curvature.
\end{corollary}

\subsection{Steady Gradient Solitons}

A soliton that evolves by scaling and also by diffeomorphism has the form $$g(t) = \sigma(t) \phi(t)^*g_0,$$ 
from which it follows that 
$$\frac{\partial}{\partial t} g = \sigma'(t)  \phi^*g_0 + \sigma(t) \phi(t)^*(L_X g_0),$$ where $X(t)$ is the vector field such that $\frac{d}{dt} \phi(t)(p) = X(\phi(t)(p)$ for every point $p \in M$.  If $X = \nabla f$ for some function $f$, then the solution is called a {\it gradient soliton}. 

\subsubsection{Reduction to a Condition on Initial Data:}

For a steady gradient soliton, it follows from the above considerations together with the steady soliton condition $\sigma(t)=1$ that at time
$t = 0$, the metric $g_0=g(0)$ and the function $f_0=f(0)$ must satisfy 
\begin{equation}
\label{gradsoliton}
Rc(g_0) + \frac{\alpha}{2} Rm^2(g_0) + \nabla^{g_0}  \nabla^{g_0}  f_0 = 0.
\end{equation}
One would like to show that, as is the case for Ricci flow, if a metric $g_0$ on $M$ and  a function $f_0:M\rightarrow R$ satisfy  (\ref{gradsoliton}), then there exists a steady gradient soliton $(M,g(t),f(t))$ with initial data $(M,g_0,f_0)$. 

The proof of this result for RG-2 flow is much like that for Ricci flow (see for example Ch. 4, \cite{CLN}). We first define $Y = \nabla_{g_0} f_0$, and set $\psi(t)$ to be the 1-parameter family of diffeomorphisms generated by 
$Y.$ We then claim that $(M, g(t) =\psi(t)^*g_0, h(t)  = \psi(t)^* (f_0))$ is a steady gradient soliton for RG-2 flow.
To see this, we calculate (for the left hand side of the RG-2 flow equation), at an arbitrary time $t_0$,
$$\frac{\partial g(t)}{\partial t}|_{t=t_0} = \frac{\partial}{\partial t}|_{t=t_0}\psi(t)^*g_0.$$
Since, by definition, $\frac{\partial}{\partial t}|_{t=t_0} \psi(t) =  \nabla_{g_0} f_0 
= \psi(t_0)_*(\nabla_{g(t_0)}h(t_0)),$ 
we obtain  (rewriting $t_0$ as $t$), 
 $$\frac{\partial}{\partial t}g(t) =  L_{\nabla_{g(t)}h(t)} g(t). $$
On the other hand (for the right hand side of the RG-2 equation), we find
\begin{align*}
-2Rc(g(t)) - \alpha Rm^2(g(t)) & = \psi(t)^*(-2Rc(g_0) - \alpha Rm^2(g_0)) \\
&=\psi(t)^*(2\nabla^{g_0}\nabla^{g_0}f_0) =\psi(t)^*(L_{\nabla_{g_0}f_0}g_0) \\
&=L_{\psi(t)^*(\nabla_{g_0}f_0)}(\psi_t^* g_0) =L_{\nabla_{\psi(t)^*(g_0)}(f_0 \cdot\psi(t))}(\psi_t^* g_0) \\
&= L_{\nabla_{g(t)}h(t)} g(t)
\end{align*}
We have thus shown that $(M, g(t) =\psi(t)^*g_0, h(t)  = \psi(t)^* (f_0))$ does satisfy the RG-2 flow equation, and is consequently a steady RG-2 gradient soliton.

\subsubsection{2-d Cigar Soliton:}
It of course remains to find solutions to condition \eqref{gradsoliton}. We now show that there is an RG-2 analog of the Ricci flow cigar soliton (known in the physics literature as Witten's black hole) and that its geometry is qualitatively similar to that of the Ricci flow version.

\begin{remark} The set-up in the proof below follows Lemma 2.7 in \cite{CK}. One can also set up the equations for the analog of the 3 dimensional   Bryant soliton, as in \cite{C}, Chapter 3 section 4. One writes the curvature tensor in terms of the Ricci curvatures as in equation  (\ref{3dfixpt}) to calculate the $Rm^2$ term.\end{remark}

\begin{remark}
In the proof of Theorem \ref{cigar}, if one instead takes $K = 0$, one obtains the flat metric. If one assumes that $K < 0$, then the constant $c$ in equation (\ref{f'}) is negative, and the solution of equation (\ref{phi''})
approximates the Ricci flow soliton of negative curvature $g = ds^2 + \psi(s)^2 d\theta^2$, where $\psi(s) = \sqrt{\frac{2}{|c|}} \tan(\sqrt{\frac{|c|}{2}}s).$
\end{remark}
  
\begin{theorem}
\label{cigar}
There exists a unique rotationally symmetric 2-dimensional steady RG-2 gradient soliton of constant positive curvature $K>0.$ 
\end{theorem}

\begin{proof}
We seek rotationally symmetric metrics of the form
 $$g = ds^2 + \phi(s)^2 d\theta^2$$ 
 together with functions $f(s)$ which satisfy the initial value 
 soliton equation 
\begin{equation} 
\label{solitoneq}
 Rc(g) + \frac{\alpha}{4} Rm^2(g) = \nabla \nabla f.
\end{equation} 
 Here $\phi(s) > 0$ and is presumed smooth.
In order for the metric to extend smoothly over the origin, we impose the boundary conditions $\phi(0) = 0,$ and $\phi'(0) = 1.$ (The fact that these conditions are necessary and sufficient for smooth extension is proven, e.g., in Lemma 2.10, pg. 29 in \cite{CK}.) We choose the orthonormal frame field 
$e_1 = \frac{\partial}{\partial s}, \ \  e_2 = \frac{1}{\phi(s)}\frac{\partial}{\partial \theta},$ and using moving frames we calculate
the curvature of the metric to be $K = -\frac{\phi''(s)}{\phi(s)}. $ We then obtain the terms on the left hand side of the soliton equation by calculating 
\begin{align*}
Rc_{11} &= Rc_{22} = K, \\
Rm_{11}^2 &= Rm^2_{1221} + Rm^2_{1212} = 2K^2,\\
Rm_{22}^2 &=Rm^2_{2112}+Rm^2_{2121} =2K^2,
\end{align*}
from which we determine that the soliton equation takes the form
\begin{equation}
\label{GenCigSol}
(K+\frac{\alpha}{2}K^2) g = \nabla \nabla f.
\end{equation}

To calculate the Hessian on the right hand side of \eqref{GenCigSol}, we use $\nabla_{e_1}e_1 = 0$ and $\nabla_{e_2} e_2 = -\frac{\phi'(s)}{\phi(s)} e_1.$ Since $g_{ij} = \delta_{ij}$ with respect to the orthonormal frame $\{e_1,e_2\}$, equation (\ref{GenCigSol}) reduces to the two equations
\begin{align}
(K+\frac{\alpha}{2}K^2)g_{11}& = (\nabla \nabla f)_{11} = e_1(e_1 f) -\nabla_{e_1}e_1 f = f''(s),\\
(K+\frac{\alpha}{2}K^2)g_{22}&=(\nabla \nabla f )_{22}= e_2(e_2 f) -\nabla_{e_2}e_2 f = \frac{\phi'(s) f'(s) }{\phi(s)}.
\end{align}
Combining these two equations we obtain the  separable ODE $f''(s) = \frac{\phi'(s)}{\phi(s)} f'(s)$, which can be integrated (with integration constant $c$) to produce 
\begin{equation}
\label{f'}
f'(s) = c \phi(s).
\end{equation} The constant $c$ cannot be zero since $K > 0$ implies $f''(s) >0$, from which we infer (noting that  $\phi>0$) that  $f'$ must have a sign. Since $K(s) = \phi' (s)f' (s)/\phi (s)> 0$, we see that   
$\phi'$ and $f'$ must have the same sign for all $s$. Since $\phi'(0) =1$,  $\phi'$ and $f'$ must be  positive.  It thus follows from (\ref{f'}) that  $c$ must be positive. The soliton equation therefore takes the form 
\begin{equation}
\label{Solitonred}
-\frac{\phi''(s)}{\phi(s)} + \frac{\alpha}{2}(\frac{\phi''(s)}{\phi(s)})^2 = c \phi'(s).
\end{equation}
Equation \eqref{Solitonred} is  quadratic in $\phi''$, so we can solve for the second derivative of $\phi$:
\begin{equation}
\label{phi''}
\phi''= \frac{\phi}{\alpha} (1 - \sqrt{1+2 c \alpha \phi'}).
\end{equation}
Here  we choose the negative root since $ \phi'' < 0.$ 
The associated equation in the Ricci flow case is 
$\psi'' +c\psi \psi' = 0,$  which has the explicit solution 
$\psi(s) = \sqrt{\frac{2}{c}} \tanh(\sqrt{\frac{c}{2}}s)$,
with the  associated curvature $K_{\psi} =c-c\tanh(\sqrt{\frac{c}{2}s})^2. $

In order to see that \eqref{phi''} has a unique solution, 
we set $v:= \frac{d \phi}{ds}$ and rewrite \eqref{phi''} as the following ODE system:
\begin{align}
\label{phisystem}
v' &= \phi'' = \frac{\phi}{\alpha}(1 - \sqrt{1+2 c \alpha \phi'}))\\
\phi' &= v, 
\end{align}
with the initial conditions
$v(0) = 1,  \ \phi(0) = 0.$
We may write this as $(v', \phi') = F(\phi, v) = (\frac{\phi}{\alpha}(1 - \sqrt{1+2 c \alpha v}),v).$

Standard well-posedness theorems for ODE systems (see e.g.\ Thm 20.9 in \cite{Olv} and its references) tell us that so long as the partial derivatives of $F(\phi, v)$ exist and are continuous in a neighborhood of the initial values, then existence and uniqueness of solutions holds for this initial value problem. We readily verify that  $F(\phi, v)$ is $C^1$ for all $v > -1/(2 c \alpha)$. 
Hence the initial value problem has a unique solution near $s=0$, for the stated initial data, $v(0) = 1,  \ \phi(0) = 0.$ Further, one can argue that solutions continue for all $s$. To show this, we first argue that so long as $\phi(s)$ stays bounded, the solution continues, with $v(s)$ decreasing and staying strictly positive. The fact  that $v(s)$ (starting at $v(0)=1)$ is a decreasing function follows immediately from \eqref{phisystem}. Now, let us presume that there is some (first) value $s_1>0$ at which $v(s_1)=0.$ It then  follows from \eqref{phisystem} that  one must have $v'(s_1)=0$. Since $F(\phi, v)$ is $C^1$ in a neighborhood of the point $(\phi_0,0)$ for any $\phi_0>0$, and since it follows that the solution $(\phi(s),v(s))=(\phi_0, 0)$ is the only solution consistent with the initial data $(\phi(s_1),v(s_1))=(\phi_0, 0)$, and finally since this solution is inconsistent with the initial data $(\phi(0),v(0))=(0, 1)$, we have a contradiction. Hence, presuming the boundedness of $\phi$, we see that  there exists a unique solution with $v(s_0)>0$, constantly decreasing.

To argue that $\phi$ remains bounded, we consider the equation 
$\frac{dv}{d\phi} = \frac{\phi}{\alpha v} (1 - \sqrt{1+2 c \alpha v})).$
This is separable, and we can integrate both sides. One can see from the solution that as $v$ decreases to zero, $\phi$ approaches a constant.  Global existence follows.

To better understand the solution $\phi$ of (\ref{phi''}), it is useful to consider the curvature 
\begin{equation}
\label{curvatures}
K= -\frac{\phi''}{\phi} =\frac{1}{\alpha}\sqrt{1+2 c\alpha \phi'}-\frac{1}{\alpha}
\end{equation} at any fixed $s$.  Expanding
$\sqrt{1+2 c \alpha \phi'} - 1$ about  $\alpha = 0,$  
we find that $$K= c\phi' - (\phi')^2 \frac{c^2}{2} \alpha + ....$$ For example  $K(0) \rightarrow c=K_\psi(0)$ as $\alpha \rightarrow 0,$ since $\phi'(0) = 1.$
Furthermore, for any $\alpha$ we have $K(s)\rightarrow 0$ as $s \rightarrow \infty.$ To see that, we recall that  $\phi' $ is positive and strictly decreasing. In fact, $\phi'$ must decrease to zero, since if it were to approach some $\epsilon > 0$, then  $\phi$ would be  increasing and $\phi''$ would be going to zero, so that $-\phi''/\phi$ would go to zero. But by equation (\ref{curvatures}),  $-\phi''/\phi$ would approach a positive constant. Hence, we must have $\phi'$ and $K$ going to zero as $s \rightarrow 0.$ In the following pictures, we compare the RF cigar soliton solution $\psi$ with the RG-2 flow cigar for small $\alpha$ (middle diagram), and for large $\alpha.$

\includegraphics[height=5cm]{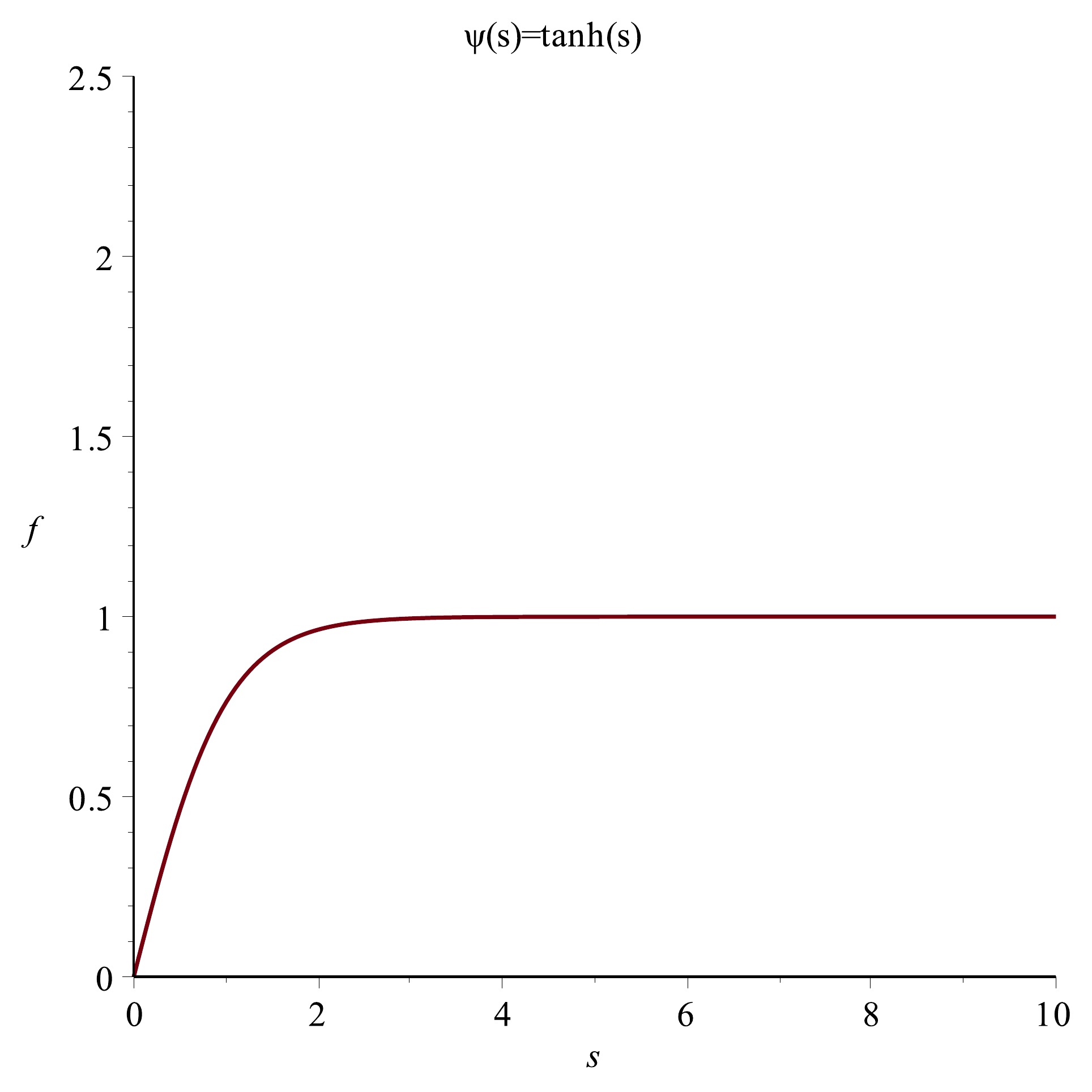}
\includegraphics[height=5cm]{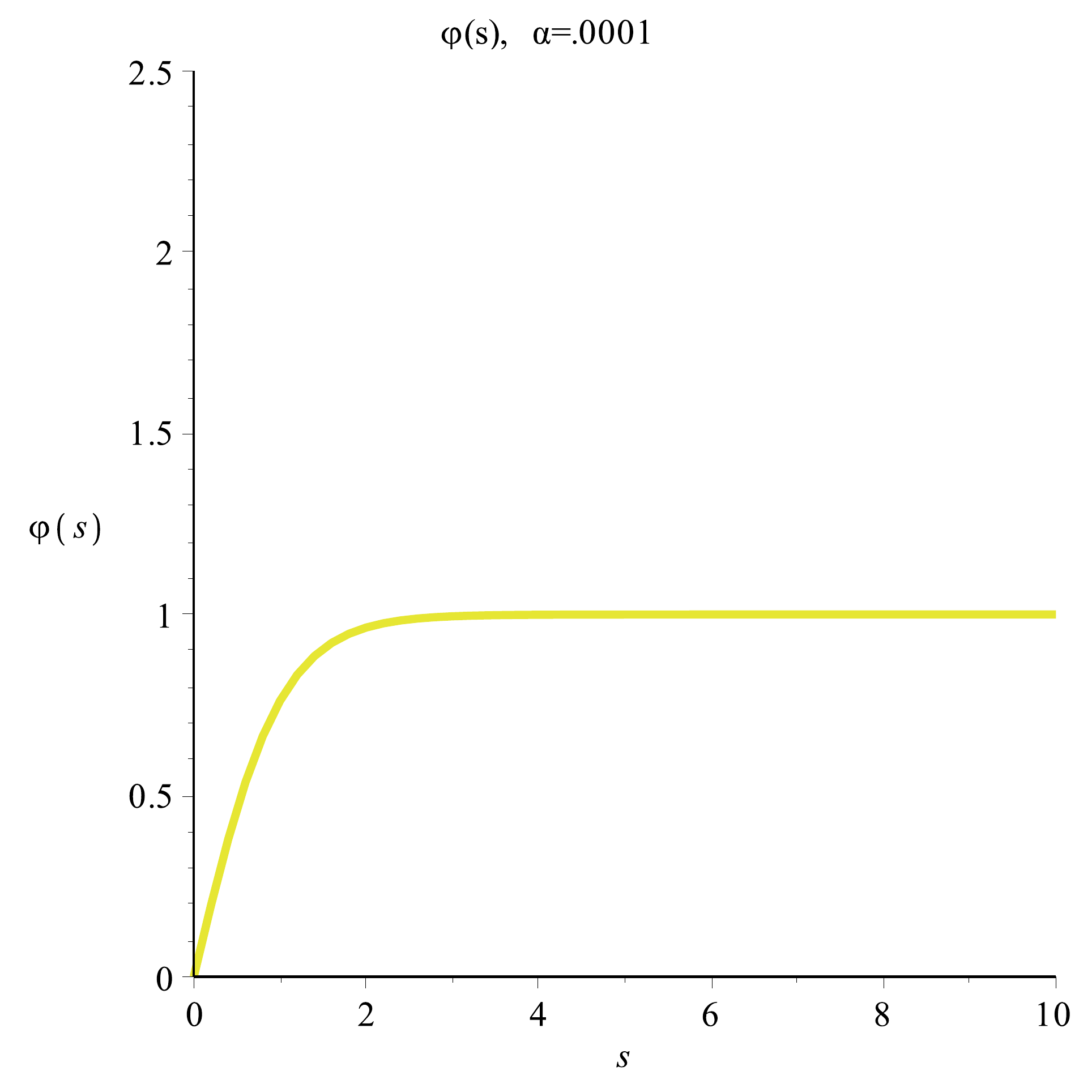}
\includegraphics[height=5cm]{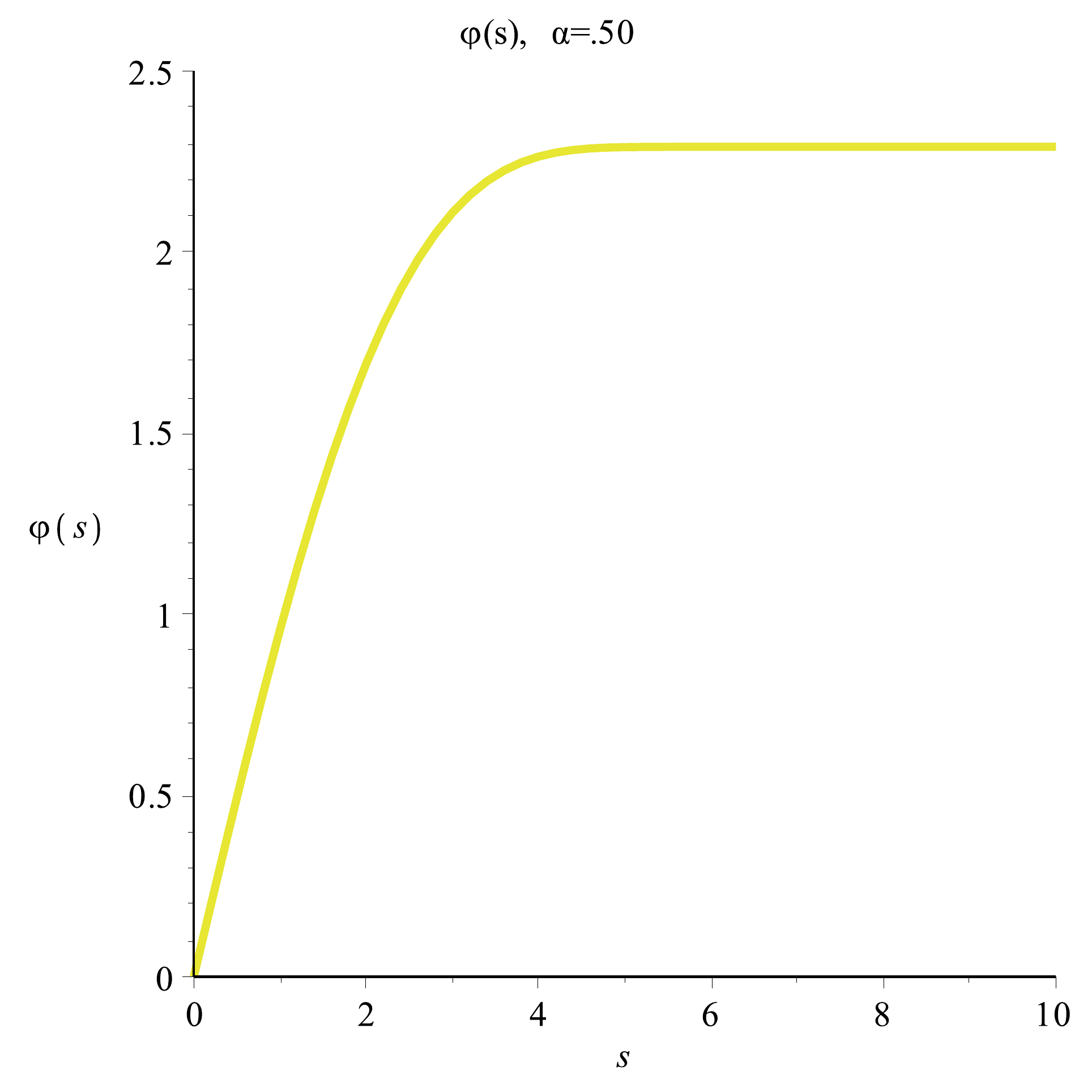}

\end{proof}

\subsection{Gaussian Soliton:}

The canonical metric $g_{Euc}$ in Euclidean space is invariant under rescaling: $c g_{Euc} = \phi^* g_{Euc}$, where $\phi = \sqrt{c} id.$ ($\mathbb R^n, g_{Euc})$ can therefore be considered an expanding, shrinking, or steady soliton for the RG-2 flow.

\section{3-d Locally Homogeneous Spaces}
\label{LocHomog}

As a final example of a class of special solutions of the RG-2 flow, we consider the 3-dimensional locally homogeneous spaces. These are spaces such that  for each pair of points $p$ and $q$, there exist neighborhoods $p \in V_p$ and $q \in V_q$ that are isometric. In this case the spatial independence of the curvature results in the flow being a system of ODEs. The Ricci flow on 3-d locally homogeneous spaces has been well-studied (see e.g. \cite{IJ}, \cite{GP}, \cite{KM}); solutions either develop finite time singularities, or exhibit characteristic ``pancake'' or ``cigar" singularities, where one (resp. two) direction shrinks and the others expand.  In \cite{GGI} we explore the question of whether  for fixed $\alpha$ and for a fixed family of geometries, the RG-2 flow has asymptotic behavior similar to that of the Ricci flow for these spaces. In the cases Sol, SU(2), and SL($2,\mathbb R$), we set two directions equal so that the system is amenable to phase plane analysis. 

The conclusions we reach are consistent with the familiar theme that the behavior depends on the size of $\alpha \times$Curvature if negative curvatures are involved.  More specifically, in the $\mathbb H^3, \mathbb H^2\times \mathbb R,$ Nil, symmetric Sol, and symmetric SL($2,\mathbb R$) cases, one can choose $\alpha$ sufficiently small so that the asymptotic behavior is similar to the Ricci flow, and $\alpha$ sufficiently large so that all directions contract, in contrast to the Ricci flow. 

For every fixed $\alpha$, one finds a curve that partitions the phase plane into two regions. Solutions whose initial conditions start in one of the regions are immortal and demonstrate either cigar or pancake asymptotics consistent with the behavior of the Ricci flow; solutions whose initial conditions start in the other develop shrinker asymptotics with singularity developing in finite time. So, for every positive value of $\alpha$, no matter how small, one can find initial data such that dichotomous behavior occurs for the geometries with negative curvatures.

\section{Open Problems}
\label{Open}

There are many interesting open problems concerning  the RG-2 flow:

\begin{enumerate}

\item Variational formulation: Based on physical results the RG-2 flow should be a gradient flow (see \cite{P} and \cite{OSW} for the Ricci flow calculation). A natural starting point would be Perelman's functional for Ricci flow, perhaps combined with the functional $\int_M |Rm|^2 d\mu$.

\item Preserved curvature conditions: Little is known about what curvature conditions are preserved along the flow. In particular it would be useful to find classes of initial data (subspaces of the space of metrics), for which the parabolicity condition $1+\alpha K_{ij} > 0$  is preserved. 

\item Singularities: We have seen that the RG-2 flow is not scale invariant, which is problematic in formulating blow-up arguments to study singularities. It would be useful to determine if blowup arguments in any form are effective for analyzing RG-2 singularities.

\item Harnack inequality: The derivation of a gradient Harnack inequality for RG-2 flow could be guided by our understanding of solitons.  

\item Additional fields: For physical reasons one might be motivated to add terms to the nonlinear sigma model action involving an antisymmetric tensor $B_{ij}$ called a B-field, or a scalar function called a dilaton. The Ricci flow coupled with additional fields has been studied by Streets in \cite{St} and Oliynyk/Suneeta/Woolgar in \cite{OSW}. 

\item Physical applications: It has been postulated that the RG-2 flow should control the existence of a continuum limit,  but we have not found a proof of this in the literature.  Briefly, one of the mathematically problematic aspects of this quantum field theory is that one needs a measure on the infinite dimensional space of all functions $\Sigma \rightarrow M$. One can``regularize" by approximating the domain with a lattice $\Sigma = \mathbb Z_N \times \mathbb Z_N$, so that the space of maps is $M^{N^2}$. The continuum limit is the limit as $N \rightarrow \infty.$ 

\item Mathematical rigor: As we note  in the introduction, the derivation of the Renormalization Group flow, and its approximation in the perturbative realm,  is not rigorous mathematically. In particular, an understanding of the measure described in problem (6) above is a crucial (difficult) outstanding problem. 

\item Geometric interest: The increased geometric complexity of the RG-2 flow compared with the Ricci flow (including the presence of   additional fixed points)  may be of use in the analysis of the geometry of 3-manifolds. 
\end{enumerate}

\end{document}